\documentclass[jmp,preprint]{revtex4-1}
\usepackage{amsmath,amssymb,amsthm,bm,graphicx,mathrsfs,url}
\usepackage[usenames,dvipsnames]{xcolor}
\usepackage[colorlinks=true,linkcolor=Red,citecolor=Green]{hyperref}
\usepackage{amsxtra}
\usepackage[final]{changes}
\newcommand{\stkout}[1]{\ifmmode\text{\sout{\ensuremath{#1}}}\else\sout{#1}\fi}

\let\epsilon=\varepsilon

\newcommand{\RR}{{\mathbb R}}
\newcommand{\FF}{{\mathcal F}}

\newcommand{\NN}{{\mathbb N}}
\newcommand{\CC}{{\mathbb C}}
\newcommand{\CI}{{{\mathcal C}^\infty}}
\newcommand{\CIb}{{{\mathcal C}^\infty_b}}
\newcommand{\CIc}{{{\mathcal C}^\infty_{\rm{c}}}}

\newtheorem{thm}{Theorem}
\newtheorem{prop}{Proposition}%[section]

\newtheorem{lem}{Lemma}

\numberwithin{equation}{section}

\DeclareMathOperator{\Det}{det}

\DeclareMathOperator{\dist}{dist}

\DeclareMathOperator{\comp}{comp}

\let\Im=\Imag

\let\Re=\Real

\DeclareMathOperator{\supp}{supp}
\DeclareMathOperator{\tr}{tr}

\draft % marks overfull lines with a black rule on the right

\begin{document}

% Use the \preprint command to place your local institutional report number
% on the title page in preprint mode.
% Multiple \preprint commands are allowed.
%\preprint{}

\title{Resonances as Viscosity Limits for Exponentially Decaying Potentials} %Title of paper

% repeat the \author .. \affiliation  etc. as needed
% \email, \thanks, \homepage, \altaffiliation all apply to the current author.
% Explanatory text should go in the []'s,
% actual e-mail address or url should go in the {}'s for \email and \homepage.
% Please use the appropriate macro for the type of information

% \affiliation command applies to all authors since the last \affiliation command.
% The \affiliation command should follow the other information.

\author{Haoren Xiong}
\email{xiong@math.berkeley.edu}
%\homepage[]{Your web page}
%\thanks{}
%\altaffiliation{}
\affiliation{Department of Mathematics, University of California,
Berkeley, CA 94720, USA}

% Collaboration name, if desired (requires use of superscriptaddress option in \documentclass).
% \noaffiliation is required (may also be used with the \author command).
%\collaboration{}
%\noaffiliation

\date{\today}

\begin{abstract}
% insert abstract here
We show that the complex absorbing potential (CAP) method for computing scattering resonances applies to the case of exponentially decaying potentials. That means that the eigenvalues of $-\Delta + V - i\epsilon x^2$, $|V(x)|\leq C e^{-2\gamma |x|}$ converge, as $ \epsilon\to 0+ $, to the poles of the meromorphic continuation of $ ( -\Delta + V -\lambda^2 )^{-1} $ uniformly on compact subsets of $\Re\lambda>0$, $\Im\lambda>-\gamma$, $\arg\lambda > -\pi/8$.
\end{abstract}

% insert suggested PACS numbers in braces on next line

\maketitle %\maketitle must follow title, authors, abstract and \pacs

% Body of paper goes here. Use proper sectioning commands.
% References should be done using the \cite, \ref, and \label commands
\section{Introduction}
\label{introduction}
The complex absorbing potential (CAP) method has been used as a computational tool for finding scattering resonances -- see Riss--Meyer \cite{RiMe} and Seideman--Miller \cite{semi} for an early treatment and Jagau et al \cite{Jag} for some recent developments. For potentials $ V \in L^\infty_{\comp} $ the method was justified by Zworski \cite{Zw-vis}. In \cite{xiong2020} the author extended it to potentials which are dilation analytic near infinity. In this paper we show that the CAP method is also valid for potentials which are exponentially decaying. While the key component of \cite{Zw-vis} and \cite{xiong2020} was the method of complex scaling (see Hunziker \cite{hunziker1986}, Sj{\"o}strand--Zworski \cite{sjostrand1991} for an account and references), here we use complex scaling on the Fourier transform side following Nakamura \cite{nakamura1990} and Kameoka--Nakamura \cite{Nakamura}.

Thus, we consider the Schr\"odinger operator $ P := -\Delta + V $ acting on $L^2(\RR^n)$ whose potential is exponentially \replaced{decaying}{decreasing}, this means that there exist $C,\gamma>0$ such that
\begin{equation}\label{V decay}
    |V(x)| \leq C e^{-2\gamma |x|}.
\end{equation}
Let $ R_V(\lambda) = ( P - \lambda^2 )^{-1} $ be the resolvent of $P$, initially defined for $\Im\lambda>0$. The exponentially weighted resolvent $ \sqrt{V} R_V(\lambda) \sqrt{V} $ can be meromorphically continued to the strip $\Im\lambda > -\gamma$, see Froese \cite{froese}, Gannot \cite{gannot} and a review in \S\ref{resonances}. Resonances of $P$ are the poles in this meromorphic continuation.

We now introduce a {\em regularized} operator,
\begin{equation}
\label{eqn:Peps}
  P_\epsilon := - \Delta - i \epsilon x^2 + V, \ \ \epsilon > 0 .
\end{equation}
(We write $ x^2 := x_1^2 + \cdots + x_n^2 $.) It is easy to see, with details reviewed in \S\ref{eigenvalues}, that $ P_\epsilon $
is a non-normal unbounded operator on $ L^2 ( \RR^n ) $ with a discrete spectrum. When $ V\equiv 0 $, $P_\epsilon$ is reduced to the rescaled Davies harmonic oscillator -- see \S\ref{davies}, whose spectrum is given by
\[
     \sqrt{\epsilon}\,e^{-i\pi/4}(2|\alpha|+n),\quad\alpha\in\NN_0^n,\quad |\alpha|:=\alpha_1 + \cdots + \alpha_n,
\]
where $\NN_0$ denotes the set of nonnegative integers.
Thus we will restrict our attentions to $\arg z > -\pi/4$. Suppose that
\begin{equation}
\label{PepsEigenvalues}
    \sigma(P_\epsilon) \cap \CC\setminus e^{-i\pi/4}[0,\infty) = \{\lambda_j(\epsilon)^2 \}_{j=1}^\infty,\quad -\pi/8 < \arg\lambda_j(\epsilon) < 7\pi/8.
\end{equation}
Zworski \cite{Zw-vis} proved that resonances can be defined as the limit points of $\{ \lambda_j(\epsilon) \}_{j=1}^\infty$ as $\epsilon\to 0+$, in the case of compactly supported potentials. We generalize this result to the case of exponentially decaying potentials. More precisely, we have

\begin{thm}
\label{thm:1}
For any $ 0 < a'< a< b $ and $\gamma'<\gamma$ such that the rectangle
\begin{equation}
\label{rectangle Omega}
    \Omega := (a',a) + i(-\gamma',b)  \Subset \{\lambda\in\CC: -\pi/8<\arg\lambda<7\pi/8\},
\end{equation}
we have, uniformly on $\Omega$,
\[   \lambda_j ( \epsilon ) \to \lambda_j , \quad \epsilon \to 0 + , \]
where $ \lambda_j $ are the resonances of $ P $.
\end{thm}

\medskip
\noindent
{\bf Notation.}
We use the following notation: $ f =  \mathcal O_\ell (
 g   )_H $ means that
$ \|f \|_H  \leq C_\ell  g $ where the norm (or any seminorm) is in the
space $ H$, and the constant $ C_\ell  $ depends on $ \ell $. When either
$ \ell $ or
$ H $ are absent then the constant is universal or the estimate is
scalar, respectively. When $ G = \mathcal O_\ell ( g ) : {H_1\to H_2 } $ then
the operator $ G : H_1  \to H_2 $ has its norm bounded by $ C_\ell g $.
Also when no confusion is likely to result, we denote the operator
$ f \mapsto g f $ where $ g $ is a function by $ g $.

\section{meromorphic continuation}
\label{resonances}

In this section we will introduce a meromorphic continuation of the weighted resolvent $ \sqrt{V} R_V(\lambda) \sqrt{V} $ from $\Im\lambda>0$ to the strip $\Im\lambda>-\gamma$ under the assumption \eqref{V decay}. As in \cite{froese}, we define the resonances of $P$ as the poles of this meromorphic continuation, with agreement of multiplicities. For a detailed presentation, we refer to \cite{froese}.

Let $R_0(\lambda):=(-\Delta-\lambda^2)^{-1}$ be the free resolvent. For $\Im\lambda>0$, the resolvent equation
\[ R_0(\lambda) - R_V(\lambda) - R_V(\lambda) V R_0(\lambda) = 0 \]
implies
\[ ( I - \sqrt{V} R_V(\lambda)\sqrt{V} ) ( I + \sqrt{V} R_0(\lambda) \sqrt{V} ) = I. \]
Since $ R_0(\lambda) = \mathcal O(|\Im\lambda|^{-1}):L^2 \to L^2 $, then for $\Im\lambda$ large, $ I + \sqrt{V} R_0(\lambda)\sqrt{V} $ is invertible by a Neumann series argument and
\begin{equation}
\label{eqn:meroCont}
    I - \sqrt{V} R_V(\lambda)\sqrt{V} = ( I + \sqrt{V} R_0(\lambda) \sqrt{V} )^{-1}.
\end{equation}
We will show that the right side of $\eqref{eqn:meroCont}$ has a meromorphic continuation. For that, we recall some bounds of the free resolvent with exponential weights, see \cite{gannot} for details, to prove the following lemma:

\begin{lem}
\label{lem:meroCont}
For any $a>0$ and $\gamma'<\gamma$,
\[ \lambda \mapsto ( I + \sqrt{V}R_0(\lambda)\sqrt{V} )^{-1},\quad \Re\lambda > a,\ \Im\lambda > -\gamma', \]
is a meromorphic family of operators on $L^2(\RR^n)$ with poles of finite rank.
\end{lem}

\begin{proof}
Choose $\varphi\in \CI(\RR^n)$ satisfying $\varphi(x)=|x|$ for large $|x|$, it is well known that for each $c>0$, the weighted resolvent:
\[ e^{-c\varphi}R_0(\lambda)e^{-c\varphi}: L^2(\RR^n) \to L^2(\RR^n) \]
extends analytically across $\Re\lambda>0$ to the strip $\Im\lambda>-c$, see \cite[\S 1]{gannot} and references given there. Moreover, Gannot \cite[\S 1]{gannot} proved that for any $a,c,\epsilon>0$ and $\alpha\in\NN^n,\ |\alpha|\leq 2$ there exists $C_\alpha=C_\alpha(a,c,\epsilon)$ such that
\begin{equation}
\label{weighted R0 bounds}
    \| D^\alpha (e^{-c\varphi}R_0(\lambda)e^{-c\varphi}) \|_{L^2\to L^2} \leq C_\alpha |\lambda|^{|\alpha|-1},\quad\textrm{for }\Im\lambda > -c+\epsilon,\ \Re\lambda > a.
\end{equation}
In particular, for $\Re\lambda > a$ and $\Im\lambda > -\gamma'$,
\[ \lambda\mapsto e^{-\gamma'\varphi}R_0(\lambda)e^{-\gamma'\varphi}
\]
is an analytic family of operators $L^2\to H^2$. Since $\lim_{|x|\to\infty}\sqrt{V(x)}e^{\gamma'\varphi(x)} = 0$ by \eqref{V decay}, it is easy to see that $\sqrt{V}e^{\gamma'\varphi}:H^2 \to L^2$ is compact. Hence,
\[ \lambda \mapsto \sqrt{V}R_0(\lambda)\sqrt{V} = \sqrt{V}e^{\gamma'\varphi} (e^{-\gamma'\varphi}R_0(\lambda)e^{-\gamma'\varphi}) \sqrt{V}e^{\gamma'\varphi} \]
is an analytic family of compact operators $L^2\to L^2$ for $\Re\lambda > a,\ \Im\lambda > -\gamma'$. Recalling that $ I + \sqrt{V}R_0(\lambda)\sqrt{V} $ is invertible for $\Im\lambda\gg 1$, then by the analytic Fredholm theory -- see \cite[\S C.4]{res}, $\lambda\mapsto (I+\sqrt{V}R_0(\lambda)\sqrt{V})^{-1}$ is a meromorphic family of operators in the same range of $\lambda$.
\end{proof}

From now on, we identify the resonances $\lambda_j$, in $\Omega$ given in \eqref{rectangle Omega}, with the poles of $(I+\sqrt{V}R_0(\lambda)\sqrt{V})^{-1}$, with agreement of multiplicities. More precisely, the multiplicity of resonance $\lambda$ is given by
\begin{equation}
\label{resonance mult}
    m(\lambda) := \frac{1}{2\pi i} \tr \oint_\lambda ( I + \sqrt{V}R_0(\zeta)\sqrt{V} )^{-1} \partial_\zeta (\sqrt{V}R_0(\zeta)\sqrt{V})\,d\zeta,
\end{equation}
where the integral is over a positively oriented circle enclosing $\lambda$ and containing no poles other than $\lambda$.

% If in two-column mode, this environment will change to single-column format so that long equations can be displayed.
% Use only when necessary.
%\begin{widetext}
%$$\mbox{put long equation here}$$
%\end{widetext}

% Figures should be put into the text as floats.
% Use the graphics or graphicx packages (distributed with LaTeX2e).
% See the LaTeX Graphics Companion by Michel Goosens, Sebastian Rahtz, and Frank Mittelbach for examples.
%
% Here is an example of the general form of a figure:
% Fill in the caption in the braces of the \caption{} command.
% Put the label that you will use with \ref{} command in the braces of the \label{} command.
%
\section{resolvent estimates for the Davies harmonic oscillator}
\label{davies}

The operator $ H_c := -\Delta + cx^2,\ -\pi < \arg c \leq 0 $, was used by Davies \cite{Dav} to illustrate properties of non-normal differential operators. We recall some known facts about $H_c$ and its resolvent. As established in \cite{Dav}, $H_c$ is an unbounded operator on $L^2(\RR^n)$ with the discrete spectrum given by
\begin{equation}
\label{Davies spectrum}
    \sigma(H_c) =  \{ c^{1/2}(n+2|\alpha|) : \alpha\in\NN_0^n \}.
\end{equation}
In particular $\sigma( H_{-i\epsilon} )\subset e^{-i\pi/4}[0,\infty)$, then one can study the resolvent of $ H_{-i\epsilon} $ outside $e^{-i\pi/4}[0,\infty)$. Unlike the normal operators, there does not exist any constant $C$ such that $ \| (-\Delta - i\epsilon x^2 - z)^{-1} \|_{L^2\to L^2} \leq C \dist( z,e^{-i\pi/4}[0,\infty) )^{-1} $. Instead, according to Hitrik--Sj\"ostrand--Viola \cite{HSV}, \cite[\S 3]{Zw-vis} and references given there, for $\Omega\Subset\{z : -\pi/2 < \arg z < 0\}\setminus e^{-i\pi/4}[0,\infty) $, there exists $ C = C(\Omega) $ such that
\begin{equation}
\label{bad Davies resolvent bound}
    \frac{1}{C} e^{\epsilon^{-\frac12}/C} \leq \| (-\Delta - i\epsilon x^2 - z)^{-1} \|_{L^2\to L^2} \leq C e^{C\epsilon^{-\frac12}},\quad z\in\Omega.
\end{equation}

In this section we will show how exponential weights dramatically improve the bound \eqref{bad Davies resolvent bound} for $(-\Delta-i\epsilon x^2 -\lambda^2)^{-1}$ in the rectangle $\Omega$ given by \eqref{rectangle Omega}, which will be crucial in the proof of Theorem \ref{thm:1}.

First, note that $ -\Delta_x - i\epsilon x^2 = \FF^{-1} ( \xi^2 + i\epsilon \Delta_\xi ) \FF $, where $\FF$ denotes the Fourier transform $\FF u(\xi) = \hat{u}(\xi) = (2\pi)^{-n/2} \int e^{-ix\cdot\xi}u(x)\,dx $. Inspired by \cite{nakamura1990} and \cite{Nakamura}, we introduce a family of spectral deformations in the Fourier space as follows.

For any fixed $\Omega$ given in \eqref{rectangle Omega}, we choose $ \rho \in \CI([0,\infty);\RR) $ with $\rho\equiv 0$ near $0$ and $\rho(t)\equiv 1$ for $t\gg 1$ such that
\begin{equation}
\label{eqn:rho}
    0\leq \rho'(t) < \gamma^{-1}\tan\frac{\pi}{8},\ \forall\, t\geq 0;\quad \Omega \Subset \{x+iy : x > 0,\ y > -\gamma\rho(x)\},
\end{equation}
and define the map
\begin{equation}
\label{eqn:psi}
    \psi : \RR^n \to \RR^n,\quad \psi(\xi)=|\xi|^{-1}\rho(|\xi|)\,\xi,
\end{equation}
 then $\psi$ is smooth with the Jacobian:
\begin{equation}
\label{eqn:Dpsi}
    D\psi (\xi) = |\xi|^{-1}{\rho(|\xi|)} I + (|\xi|^{-2}\rho'(|\xi|) - |\xi|^{-3}\rho(|\xi|))\, \xi\cdot\xi^T.
\end{equation}
Let $A$ be an orthogonal matrix with $n$-th column $|\xi|^{-1} \xi$, then we have
\begin{equation}
\label{eqn:diag Dpsi}
    A^T D\psi (\xi)\, A = \textrm{diag}[\, |\xi|^{-1}{\rho(|\xi|)},\cdots,|\xi|^{-1}{\rho(|\xi|)},\,\rho'(|\xi|) \,] .
\end{equation}
For $\theta\in\RR$, we consider a family of deformations:
\begin{equation}
\label{eqn:phitheta}
    \phi_\theta (\xi) = \xi + \theta\psi(\xi),
\end{equation}
and the corresponding unitary operators $U_\theta,\ \theta\in\RR$ defined by
\begin{equation}
\label{eqn:Utheta}
    U_\theta u(\xi) := (\Det D\phi_\theta(\xi))^{\frac12}u(\phi_\theta(\xi)).
\end{equation}
Using \eqref{eqn:diag Dpsi}, we can compute $\Det D\phi_\theta(\xi)$ explicitly, i.e.
\begin{equation}
\label{eqn:Jtheta}
    J_\theta(\xi) \equiv \Det D\phi_\theta(\xi) = \Det ( I + \theta D\psi(\xi) ) = ( 1 + \theta\rho'(|\xi|)\,)\, ( 1 + \theta |\xi|^{-1}{\rho(|\xi|)}\, )^{n-1} ,
\end{equation}
then by \eqref{eqn:rho}, $U_\theta$ is invertible as $\Det D\phi_\theta(\xi)\neq 0$ for $\theta\in\RR$, $|\theta|<\gamma$, the inverse is given by
\begin{equation}
\label{eqn:Utheta Inverse}
    U_\theta^{-1} v(\xi) = (\Det D\phi_\theta (\phi_\theta^{-1}(\xi)))^{-\frac12} v (\phi_\theta^{-1}(\xi)).
\end{equation}

Now we consider the deformed operators of $ \xi^2 + i\epsilon\Delta_\xi $:
\begin{equation}
\label{eqn:Qepstheta}
\begin{split}
    Q_{\epsilon,\theta} & := U_\theta ( \xi^2 + i\epsilon\Delta_\xi ) U_\theta^{-1}  \\
        & = \phi_\theta(\xi)^2 - i\epsilon J_\theta(\xi)^{-\frac12} D_{\xi_l} J^{lj}(\xi) J_\theta(\xi) J^{kj}(\xi) D_{\xi_k} J_\theta(\xi)^{-\frac12}
\end{split}
\end{equation}
where $D_{\xi_k}=-i\partial_{\xi_k}$, $J_\theta(\xi)=\Det D\phi_\theta (\xi)$, $ J^{lj}(\xi) = [D\phi_\theta(\xi)^{-1}]_{jl}$. To extend $Q_{\epsilon,\theta}$ to $\theta\in\CC$, we define
\begin{equation}
\label{Dgamma}
    D_\gamma := \{\theta\in\CC : |\Re\theta| + |\Im\theta| < \gamma \}.
\end{equation}
In view of \eqref{eqn:rho} and \eqref{eqn:Jtheta}, $D\phi_\theta^{-1}$ and $ \Det D\phi_\theta $ extend analytically to $\theta\in D_\gamma$. Therefore, we obtain that $Q_{\epsilon,\theta}$, given by the second equation in \eqref{eqn:Qepstheta}, extends analytically to $\theta\in D_\gamma$.

Then we introduce some preliminary results about the spectrum of $Q_{\epsilon,\theta}$ :
\begin{prop}
\label{prop:Spec_Qepstheta}
There exists constant $ \epsilon_0 = \epsilon_0 (\Omega,\gamma) $ such that for all $0< \epsilon < \epsilon_0$ and $\theta\in D_\gamma$,
\[  \sigma (Q_{\epsilon,\theta}) \cap \{ z\in\CC : |z| > 1, \pi/2 < \arg z < \pi \} = \emptyset.
\]
\end{prop}

\begin{proof}
We note that for $\theta\in D_\gamma $, by \eqref{eqn:rho},
\[ 1-\tan\frac{\pi}{8} < 1-|\theta||\rho'(t)| \leq | 1 + \theta\rho'(t) | \leq 1+|\theta||\rho'(t)| < 1 + \tan\frac{\pi}{8},\quad\forall\,t\geq 0. \]
Thus, \eqref{eqn:Jtheta} implies that $ C^{-1} < |J_\theta(\xi)| < C $ for some constant $C>0$. Since
\[ [D\phi_\theta(\xi)]_{jl} = \left( 1+\theta\frac{\rho(|\xi|)}{|\xi|} \right)\delta_{jl} + \frac{\theta|\xi|\rho'(|\xi|)-\theta\rho(|\xi|)}{|\xi|^3}\xi_j \xi_l \]
by \eqref{eqn:Dpsi}, and $\rho'\in \CIc((0,\infty))$, together with \eqref{eqn:Jtheta}, we conclude that
\begin{equation}
\label{CIbs}
    J_\theta, J_\theta^{-1}, J^{lj}\in\CIb(\RR^n),\ 1\leq j,l\leq n.
\end{equation}
Here $\CIb(\RR^n):=\{u\in\CI(\RR^n) : |\partial^\alpha u|\leq C_\alpha\textrm{ for all }\alpha\in\NN_0^n\}$. Hence we have
\begin{equation}
\label{eqn:rewrite_Qepstheta}
    Q_{\epsilon,\theta} = \phi_\theta(\xi)^2 - i\epsilon J^{kj}(\xi)J^{lj}(\xi) D_{\xi_k}D_{\xi_l} + \epsilon a_j(\xi)D_{\xi_j} + \epsilon b(\xi) ,
\end{equation}
where $a_j, b\in\CIb(\RR^n)$. Let $h=\sqrt{\epsilon}$, then $Q_{\epsilon,\theta}=q_\theta(\xi,h D_\xi;h)$ is a semiclassical differential operator -- see Zworski \cite[\S 4]{Zw}, with the symbol
\begin{equation}
\label{eqn:symbol qtheta}
    q_\theta(\xi,\xi^\ast;h) = \phi_\theta(\xi)^2 - i (D\phi_\theta(\xi)^{-2}\xi^\ast) \cdot \xi^\ast + ha_j(\xi)\xi_j^\ast + h^2 b(\xi) ,
\end{equation}
where $(\xi,\xi^\ast)$ are coordinates of $\textrm{T}^\ast \RR^n$, $D\phi_\theta(\xi)^{-2}=(D\phi_\theta(\xi)^{-1})^T (D\phi_\theta(\xi)^{-1})$ since $D\phi_\theta(\xi)$ is a symmetric matrix. Choose $m(\xi,\xi^\ast)=1+\xi^2+{\xi^\ast}^2$ as an order function, we recall the symbol class $S(m)$ from \cite[\S 4.4]{Zw},
\begin{equation}
\label{eqn:symbol class}
    S(m) := \{ a\in\CI : |\partial^\alpha a|\leq C_\alpha m\quad \textrm{for }\forall\,\alpha\in \NN_0^{2n} \}.
\end{equation}
Then by \eqref{eqn:rho}, \eqref{eqn:phitheta} and \eqref{CIbs}, we have $q_\theta\in S(m)$. Hence it suffices to show that there exists constant $h_0>0$ such that for $h<h_0$,
\[ q_\theta - z \textrm{ is elliptic in }S(m) \textrm{ for }|z|>1,\ \pi/2 < \arg z < \pi . \]
For a detailed introduction of general elliptic theory, we refer to \cite[\S 4]{Zw}.

Using \eqref{eqn:psi} we calculate:
\begin{equation}
\label{eqn:phitheta^2}
    \phi_\theta(\xi)^2 = (\xi + \theta \psi(\xi) )\cdot (\xi + \theta\psi(\xi)) = (|\xi| + \theta \rho(|\xi|))^2.
\end{equation}
Then for $\theta\in D_\gamma$, by \eqref{eqn:rho}, we have
\begin{equation}
\label{phitheta^2 bounds}
    -\pi/4 < \arg \phi_\theta(\xi)^2 < \pi/4,\quad |\phi_\theta(\xi)^2| > \left(1-\tan\frac{\pi}{8}\right)^2 |\xi|^2.
\end{equation}
To obtain similar bounds for the argument and modulus of $(D\phi_\theta(\xi)^{-2}\xi^\ast)\cdot\xi^\ast$, we recall \eqref{eqn:diag Dpsi} to compute
\begin{equation}
\label{eqn:xi*term}
    (D\phi_\theta^{-2}\xi^\ast)\cdot\xi^\ast =  ( 1+\theta{\rho(|\xi|)}{|\xi|}^{-1} )^{-2} ({\eta_1^\ast}^2 + \cdots + {\eta_{n-1}^\ast}^2 ) + (1+\theta\rho'(|\xi|))^{-2} {\eta_n^\ast}^2 ,
\end{equation}
where $\eta^\ast = A^T \xi^\ast\in\RR^n$ with the same orthogonal matrix $A$ as in \eqref{eqn:diag Dpsi}. By \eqref{eqn:rho}, for $\theta\in D_\gamma$, we have
\[ \pm \Im\theta \geq 0 \implies 0 \leq \pm\arg (1+\theta{\rho(|\xi|)}{|\xi|}^{-1}),\, \pm\arg (1+\theta\rho'(|\xi|)) < \pi/8, \]
Hence, for all $\theta\in D_\gamma$,
\begin{equation}
\label{eqn:arg xi*term}
    \pm \Im\theta \geq 0 \implies 0 \leq \mp \arg\, (D\phi_\theta^{-2}\xi^\ast)\cdot\xi^\ast < \pi/4,
\end{equation}
and by applying the following basic inequality with \eqref{eqn:rho} to \eqref{eqn:xi*term},
\begin{equation}
\label{cosine law}
    |r_1 e^{i\theta_1} + r_2 e^{i\theta_2}|^2 = r_1^2 + r_2 ^2 + 2r_1 r_2 \cos(\theta_1 - \theta_2)\geq \frac{1-|\cos(\theta_1 - \theta_2)|}{2}(r_1 + r_2)^2 ,
\end{equation}
we also obtain that for all $\theta\in D_\gamma$,
\begin{equation}
\label{eqn:norm xi*term}
    | (D\phi_\theta^{-2}\xi^\ast)\cdot\xi^\ast|\geq C|\eta^\ast|^2 = C |\xi^\ast|^2 .
\end{equation}
Since $\arg (\phi_\theta(\xi)^2 - z)\in (-\pi/2,\pi/4)$ for $ \pi/2 < \arg z < \pi $ and $ \arg -i(D\phi_\theta^{-2}\xi^\ast)\cdot\xi^\ast \in (-3\pi/4, -\pi/4) $ by \eqref{eqn:arg xi*term}, using \eqref{cosine law} together with \eqref{phitheta^2 bounds} and \eqref{eqn:norm xi*term}, we have
\begin{equation}
\label{eqn:q_theta ellipticity}
    \begin{split}
        |\phi_\theta(\xi)^2 - z - i(D\phi_\theta^{-2}\xi^\ast)\cdot\xi^\ast| &\geq C |\phi_\theta(\xi)^2 - z| +  C|- i(D\phi_\theta^{-2}\xi^\ast)\cdot\xi^\ast| \\
        & \geq C|\phi_\theta(\xi)^2| + C |z| + C |\xi^\ast |^2 \\
        & \geq C (1 + |\xi|^2 + |\xi^\ast|^2) = Cm .
    \end{split}
\end{equation}
Then by \eqref{eqn:symbol qtheta}, we conclude that there exists $h_0 > 0$ such that for all $h<h_0$, $|q_\theta - z|\geq Cm$, which completes the proof.
\end{proof}

\begin{figure}
\includegraphics[width=5in]{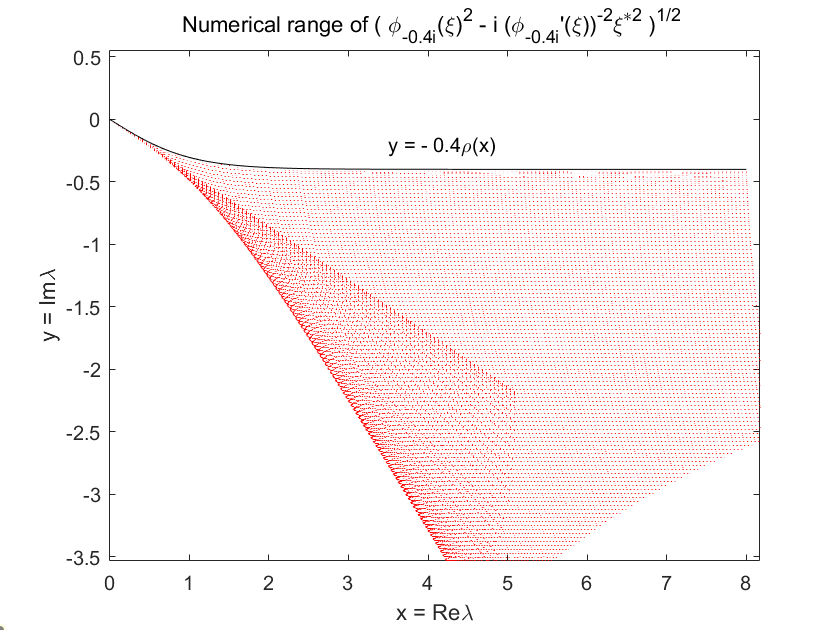}
\caption{An illustration of the results of Proposition \ref{prop:SpecQeps_beta} in the case of $\dim = 1$, $\beta=0.4$, which shows that the numerical range of the principal symbol of $Q_{\epsilon,-0.4i}$ avoids the region $\{ \lambda^2:\lambda\in\Omega \}$. We choose $\rho(\cdot)=0.4\tanh(\cdot)$ to compute the numerical range of $( \phi_{-0.4i}(\xi)^2 - i(\phi_{-0.4i}'(\xi))^{-2}{\xi^\ast}^2 )^{1/2}$.}%
\end{figure}

\begin{prop}
\label{prop:SpecQeps_beta}
For any $\beta\in (\gamma',\gamma)$ satisfying
\begin{equation}
\label{beta condition}
    \Omega \Subset \{ x+iy : x>0, y > -\beta \rho(x) \},
\end{equation}
there exists $\epsilon_0=\epsilon_0 (\Omega,\gamma,\beta)$ such that for all $0<\epsilon<\epsilon_0$,
\[
\sigma(Q_{\epsilon,-i\beta}) \cap \{ \lambda^2 : \lambda\in\Omega \} = \emptyset.  \]
\end{prop}

\begin{proof}
As in the proof of Proposition \ref{prop:Spec_Qepstheta}, it suffices to show that there exists $h_0=h_0(\Omega,\gamma,\beta)$ such that for $0<h<h_0$,
\[ q_{-i\beta}(\xi,\xi^\ast;h) - \lambda^2 \textrm{ is elliptic in }S(m)\ \textrm{for }\lambda\in\Omega. \]
Recalling $\arg -i(D\phi_{-i\beta}^{-2}\xi^\ast)\cdot\xi^\ast \in [-\pi/2,-\pi/4) $ by \eqref{eqn:arg xi*term}, in order to apply \eqref{cosine law}, we claim that
\begin{equation}
\label{eqn:argument claim}
    \exists\,\delta > 0 \textrm{ s.t. }\arg (\phi_{-i\beta}(\xi)^2 -\lambda^2 ) \leq \pi/2-\delta \textrm{ or }\geq 3\pi/4+\delta,\ \textrm{for all }\lambda\in \Omega, \,\xi\in\RR^n.
\end{equation}
We notice that for $|\xi|\gg 1$, $\phi_{-i\beta}(\xi)^2 = (|\xi|-i\beta)^2$ by \eqref{eqn:phitheta^2}, thus $\arg (\phi_{-i\beta}(\xi)^2-\lambda^2) \in (-\pi/4,0)$, in other words, there exists some large $R$ such that \eqref{eqn:argument claim} holds for $|\xi|>R$ with $\delta=\pi/2$. It remains to show that \eqref{eqn:argument claim} holds for all $|\xi|\leq R$ and $\lambda\in\Omega$. We argue by contradiction: if it does not hold, there must exist $\lambda\in\overline{\Omega}$, $\xi\in\RR^n$ such that $\arg(\phi_{-i\beta}(\xi)^2 - \lambda^2)\in [\pi/2,3\pi/4]$, i.e.
\[ 0\leq -\Re\, ( (|\xi|-i\beta\rho(|\xi|))^2 - \lambda^2 ) \leq \Im\,( (|\xi|-i\beta\rho(|\xi|))^2 - \lambda^2 ), \]
which immediately implies $\Im\lambda\leq 0$. Let $t=|\xi|$ and write $\lambda = x - iy$, then we have
\begin{align}
     x^2 -y^2 -t^2 +\beta^2 \rho(t)^2 & \leq 2xy - 2\beta t\rho(t) \label{ineq:Re} \\
     \beta t\rho(t) & \leq xy \label{ineq:Im}
\end{align}
Since $x>0$ and $0\leq y<\beta\rho(x)$ by \eqref{beta condition}, then \eqref{ineq:Re} implies that
\[ x^2 - 2\beta x\rho(x) - \beta^2 \rho(x)^2 < t^2 - 2\beta t\rho(t) - \beta^2 \rho(t)^2. \]
Let $ S(x) = x^2 - 2\beta x\rho(x) - \beta^2 \rho(x)^2 $, by \eqref{eqn:rho},
\[
\begin{split}
S'(x) & = 2x\left(1 - \beta \frac{\rho(x)}{x} - \beta\rho'(x) - \beta \frac{\rho(x)}{x} \cdot\beta\rho'(x)\right) \\
      & > 2x \left( 1-2\tan\frac{\pi}{8}-\tan^2\frac{\pi}{8} \right) = 0,
\end{split}
\]
thus $ S(x) < S(t) \implies x < t $. Recalling that $\rho$ is non-decreasing, we have $ \beta t\rho(t) \geq \beta x \rho (x) > xy $,
which contradicts \eqref{ineq:Im}. Hence \eqref{eqn:argument claim} holds, using \eqref{cosine law} and \eqref{eqn:norm xi*term}, we obtain that
\[
|\phi_{-i\beta}(\xi)^2 - \lambda^2 - i(D\phi_{-i\beta}^{-2}\xi^\ast)\cdot \xi^\ast|
\geq C(\delta) ( | (|\xi|-i\beta\rho(|\xi|))^2 - \lambda^2 | + |\xi^\ast|^2 ).
\]
Since for $|\xi|\gg 1$,
\[ | (|\xi|-i\beta\rho(|\xi|))^2 - \lambda^2 | = | (|\xi| - i\beta)^2 - \lambda^2 | \geq |\xi|^2 - \beta^2 - |\lambda|^2, \]
there exists $R=R(\Omega,\beta)>0$ such that $ | (|\xi|-i\beta\rho(|\xi|))^2 - \lambda^2 | \geq ( 1 + |\xi|^2 )/2 $ whenever $|\xi|>R$.
We also note that, by \eqref{beta condition},
\[ \textrm{dist}\,( \{ t-i\beta\rho(t) : t\geq 0 \},\,\pm\Omega ) \geq C=C(\Omega,\gamma,\beta) > 0, \]
thus $ | (|\xi|-i\beta\rho(|\xi|))^2 - \lambda^2 |\geq C^2 \geq C^2 (1+R^2)^{-1} (1 + |\xi|^2)$  for $ |\xi|\leq R $. Hence $|\phi_{-i\beta}(\xi)^2 - \lambda^2 - i(D\phi_{-i\beta}^{-2}\xi^\ast)\cdot \xi^\ast| \geq C (1 + |\xi|^2 +|\xi^\ast|^2 )$, where $C$ determined by $\Omega,\gamma,\beta$. Then by \eqref{eqn:symbol qtheta}, we conclude that there exist $h_0=h_0(\Omega,\gamma,\beta)$ and $C=C(\Omega,\gamma,\beta)>0$ such that
\begin{equation}
\label{eqn:qbeta elliptic}
    \textrm{for all }0<h<h_0,\, \lambda\in\Omega,\quad |q_{-i\beta}(\xi,\xi^\ast;h)-\lambda^2|\geq Cm(\xi,\xi^\ast),
\end{equation}
which completes the proof.
\end{proof}

Now we state the main result of this section:
\begin{lem}
\label{lem:weighted Davies resolvent estimate}
For any $0<a'<a<b$ and $\gamma'<\gamma$ such that the rectangle
\[    \Omega := (a',a) + i(-\gamma',b)  \Subset \{\lambda\in\CC: -\pi/8<\arg\lambda<7\pi/8\},
\]
there exist constant $C = C(\Omega,\gamma)>0$ and $\epsilon_0 = \epsilon_0(\Omega,\gamma) >0$ such that uniformly for $0<\epsilon<\epsilon_0$,
\[
    \| e^{-\gamma|x|} (-\Delta-i\epsilon x^2-\lambda^2)^{-1} e^{-\gamma|x|} \|_{L^2\to L^2} \leq C,\quad\forall\,\lambda\in\Omega.
\]
\end{lem}

\begin{proof}
We consider the matrix element
\[ B^\epsilon_{f,g}(\lambda) := \langle e^{-\gamma|x|} (-\Delta-i\epsilon x^2-\lambda^2)^{-1} e^{-\gamma|x|}f , g \rangle_{L_x^2},\quad\textrm{for }f,g\in L^2(\RR^n), \]
where $ \langle u,v \rangle_{L_x^2} = \int_{\RR^n} u\Bar{v} \,dx $ is the standard $L^2$ inner product. It suffices to show that there exist $C,\epsilon_0$ such that uniformly for $0<\epsilon<\epsilon_0$,
\begin{equation}
\label{eqn:bound matrix element}
| B^\epsilon_{f,g}(\lambda) | \leq C \|f\|_{L^2} \|g\|_{L^2} ,\quad \textrm{for all }f,g\in L^2,\  \lambda\in\Omega.
\end{equation}
Recalling \eqref{Davies spectrum}, both $-\Delta_x-i\epsilon x^2 -\lambda^2$ and $\xi^2 +i\epsilon\Delta_\xi -\lambda^2$ are invertible for $\lambda\in\Omega$. Then we have
\begin{equation}
\label{eqn:Bfg Fourier}
\begin{split}
    B^\epsilon_{f,g}(\lambda) & = \langle  \,( -\Delta_x - i\epsilon x^2 - \lambda^2                            )^{-1} e^{-\gamma|x|} f ,\, e^{-\gamma|x|} g                             \rangle_{L_x^2}  \\
                     & = \langle \FF^{-1} ( \xi^2 + i\epsilon \Delta_\xi - \lambda^2 )^{-1} \FF e^{-\gamma|x|} f ,\,
                     e^{-\gamma|x|} g \rangle_{L_x^2}  \\
                     & = \langle \,( \xi^2 + i\epsilon \Delta_\xi - \lambda^2 )^{-1} \FF (e^{-\gamma|x|} f) (\xi) ,\,
                     \FF (e^{-\gamma|x|} g) (\xi) \rangle_{L_\xi^2} .
\end{split}
\end{equation}
Let $ F_\gamma(\xi) := \FF (e^{-\gamma|x|} f) (\xi) $ and $ G_\gamma(\xi) := \FF (e^{-\gamma|x|} g) (\xi) $, recalling the formula
\[ \FF(e^{-|x|})(\xi) = c_n ( 1 + \xi^2 )^{-\frac{n+1}{2}},\quad c_n =            (2\pi)^{\frac{n}{2}} \Gamma((n+1)/2) \pi^{-\frac{n+1}{2}} , \]
then $ F_\gamma = K_\gamma \ast \hat{f} $ and $ G_\gamma = K_\gamma \ast \hat{g} $, where $K_\gamma(\xi) = c_n\gamma\,(\gamma^2 + \xi^2)^{-\frac{n+1}{2}}$.

First we consider, for $\theta\in\RR$, $|\theta|<\gamma$ and $U_\theta$ defined by \eqref{eqn:Utheta}, the integral kernel of the map $U_\theta\circ (K_\gamma\, \ast\ ) $:
\[  K(\xi,\eta;\theta) := (\Det D\phi_\theta(\xi))^{\frac12} K_\gamma(\phi_\theta(\xi)-\eta) ,\quad \xi,\eta\in\RR^n. \]
We claim that $K(\xi,\eta;\theta)$ has an analytic extension to $\theta\in D_\gamma$. Since $K_\gamma$ extends analytically to the strip $\{ \xi\in\CC^n : |\Im\xi|<\gamma \}$, it suffices to show that $|\Im (\phi_\theta(\xi) - \eta)| = |\Im\theta\psi(\xi)| <\gamma$, which is a direct consequence of $\theta\in D_\gamma$ and $|\psi(\xi)|\leq 1$ by \eqref{eqn:psi}. Then for $\theta\in D_\gamma$, using \eqref{eqn:rho} and \eqref{eqn:Jtheta}, we can estimate $K(\xi,\eta;\theta)$ as follows:
\[
    \begin{split}
        |K(\xi,\eta;\theta)|
        & \leq C\gamma\,|\gamma^2 + (\xi   +\theta\psi(\xi)-\eta)^2|^{-\frac{n+1}{2}} \\
        & \leq C\gamma\,|\gamma^2 - |\Im\theta|^2 |\psi(\xi)|^2 + (\xi-\eta+\Re\theta\psi(\xi))^2 |^{-\frac{n+1}{2}} \\
        & \leq C\gamma\,( \gamma^2 - |\Im\theta|^2 + (|\xi -\eta|-|\Re\theta|)^2 )^{-\frac{n+1}{2}}
    \end{split}
\]
thus
\begin{equation}
\label{eqn:intKernel estimate}
    \begin{split}
    { }& \quad  \max \,\{\,\sup_{\xi\in\RR^n}\int_{\RR^n} |K(\xi,\eta;\theta)|d\eta,\ \sup_{\eta\in\RR^n}\int_{\RR^n} |K(\xi,\eta;\theta)|d\xi\,\}  \\
    & \leq C\gamma\int_{x\in\RR^n}( \gamma^2 - |\Im\theta|^2 + (|x|-|\Re\theta|)^2 )^{-\frac{n+1}{2}} dx \leq C(\gamma,\theta) .
    \end{split}
\end{equation}
Hence, by Schur's criterion, $U_\theta\circ(K_\gamma\,\ast\ )$, first defined for $\theta\in D_\gamma\cap\RR$, with the integral kernel $K(\xi,\eta;\theta)$, extends to $\theta\in D_\gamma$ as an analytic family of operators $L^2 \to L^2$. In particular,
\[  D_\gamma\owns \theta \mapsto U_\theta F_\gamma = U_\theta(K_\gamma\ast\hat{f}) \textrm{ and }  U_\theta G_\gamma = U_\theta(K_\gamma\ast\hat{g}), \]
are two analytic families of functions in $L^2(\RR^n)$.

Now we define
\[ B^\epsilon_{f,g}(\lambda;\theta) = \langle\,(Q_{\epsilon,\theta}-\lambda^2)^{-1} U_\theta F_\gamma,U_{\bar\theta}G_\gamma \rangle\]
for $\theta\in D_\gamma$, with $Q_{\epsilon,\theta}$ given by \eqref{eqn:Qepstheta}, where we write $U_{\bar\theta}G_\gamma$ instead of $U_\theta G_\gamma$. Then by Proposition \ref{prop:Spec_Qepstheta}, there exists $\epsilon_0=\epsilon_0(\Omega,\gamma)$ such that for all $0<\epsilon<\epsilon_0$, and $|\lambda|> 1$ with $\pi/4<\arg\lambda<\pi/2$,
\[ D_\gamma\owns \theta \mapsto B^\epsilon_{f,g}(\lambda;\theta) \textrm{ is analytic. } \]
However, for $\theta\in D_\gamma\cap\RR$, since $U_\theta$ is unitary, by \eqref{eqn:Bfg Fourier} we have
\[
    \begin{split}
        B_{f,g}^\epsilon (\lambda;\theta) & = \langle U_\theta (\xi^2 + i\epsilon\Delta_\xi - \lambda^2)^{-1} U_\theta^{-1} U_\theta F_\gamma,\,U_\theta G_\gamma \rangle \\
        & = \langle U_\theta (\xi^2 + i\epsilon\Delta_\xi - \lambda^2)^{-1}  F_\gamma,\,U_\theta G_\gamma \rangle \\
        & = \langle\, (\xi^2 + i\epsilon\Delta_\xi - \lambda^2)^{-1}  F_\gamma,\, G_\gamma \rangle = B_{f,g}^\epsilon(\lambda) .
    \end{split}
\]
Thus by analyticity, $B_{f,g}^\epsilon(\lambda;\theta)\equiv B_{f,g}^\epsilon(\lambda),\ \forall\,\theta\in D_\gamma$ whenever $|\lambda|>1$, $\pi/4 < \arg\lambda < \pi/2$. In particular, for fixed $\beta\in(\gamma',\gamma)$ satisfying \eqref{beta condition},
\[ B_{f,g}^\epsilon (\lambda) = B_{f,g}^\epsilon (\lambda;-i\beta) \textrm{ whenever }|\lambda|>1,\,\pi/4 < \arg\lambda < \pi/2 . \]
In view of Proposition \ref{prop:SpecQeps_beta} and \eqref{Davies spectrum}, both $ B_{f,g}^\epsilon (\lambda) $ and $ B_{f,g}^\epsilon (\lambda;-i\beta) $ are analytic in $\Omega$. Without loss of generality, we may assume that $a>1$ in \eqref{rectangle Omega}, then
\[ \Omega\cap \{\lambda : |\lambda|>1, \pi/4 < \arg\lambda < \pi/2 \} \neq \emptyset, \]
where $B_{f,g}^\epsilon(\lambda) $ and $B_{f,g}^\epsilon(\lambda;-i\beta)$ coincide. Hence by analyticity, we conclude that for each $0<\epsilon<\epsilon_0$,
\begin{equation}
\label{eqn:Bfg-ibeta}
    B_{f,g}^\epsilon(\lambda) = B_{f,g}^\epsilon(\lambda;-i\beta) = \langle\,(Q_{\epsilon,-i\beta}-\lambda^2)^{-1} U_{-i\beta} F_\gamma,\, U_{i\beta} G_\gamma\, \rangle,\quad\forall\,\lambda\in\Omega .
\end{equation}
By the elliptic theory of semiclassical differential operators -- see \cite[\S 4.7]{Zw}, \eqref{eqn:qbeta elliptic} implies that there exists $\epsilon_0=\epsilon_0(\Omega,\gamma,\beta)$ such that for all $0<\epsilon<\epsilon_0$,
\begin{equation}
\label{eqn:Qeps-ibeta resolvent bound}
    \|(Q_{\epsilon,-i\beta}-\lambda^2)^{-1}\|_{L^2\to L^2} \leq C(\Omega,\gamma,\beta),\quad\forall\,\lambda\in\Omega.
\end{equation}
Recalling \eqref{eqn:intKernel estimate}, by Schur's criterion, we obtain that
\begin{equation}
\label{eqn:Uibeta bound}
    \begin{split}
        \|U_{-i\beta}F_\gamma\|_{L^2} & = \|U_{-i\beta}\circ(K_\gamma\ast \hat{f})\|_{L^2}\leq C(\gamma,\beta) \|\hat{f}\|_{L^2} = C(\gamma,\beta)\|f\|_{L^2} \\
        \|U_{i\beta}G_\gamma\|_{L^2} & = \|U_{i\beta}\circ(K_\gamma\ast \hat{g})\|_{L^2}\leq C(\gamma,\beta) \|\hat{g}\|_{L^2} = C(\gamma,\beta)\|g\|_{L^2}
    \end{split}
\end{equation}
Combining \eqref{eqn:Bfg-ibeta}, \eqref{eqn:Qeps-ibeta resolvent bound} and \eqref{eqn:Uibeta bound}, also noticing that $\beta$ can be determined by $\Omega, \gamma$, we obtain \eqref{eqn:bound matrix element} with $C=C(\Omega,\gamma)$, which completes the proof.
\end{proof}

\section{eigenvalues of the regularized operator}
\label{eigenvalues}

In this section we will review the meromorphy of the resolvent
\[
R_{V,\epsilon}(\lambda) := (P_\epsilon - \lambda^2)^{-1},\quad\epsilon>0,
\]
with $P_\epsilon$ in \eqref{eqn:Peps}, in a similar form to the meromorphic continuation of the weighted resolvent $\sqrt{V}R_V(\lambda)\sqrt{V}$ given by \eqref{eqn:meroCont}.

First we write $R_\epsilon(\lambda) := (-\Delta - i\epsilon x^2 - \lambda^2)^{-1}$ and recall
\begin{equation}
\label{Reps bound}
    R_\epsilon(\lambda) = \mathcal{O}_\delta(1/|\lambda|): L^2 \to L^2,\quad \delta < \arg\lambda < 3\pi/4 - \delta,\ |\lambda|>\delta,
\end{equation}
which follows from (semiclassical) ellipticity. Then
\begin{equation}
\label{eqn:RVeps}
    (P_\epsilon - \lambda^2) R_\epsilon(\lambda) = I + V R_\epsilon(\lambda), \quad -\pi/8 < \arg\lambda <7\pi/8 .
\end{equation}
In view of \eqref{Reps bound}, $I + V R_\epsilon(\lambda)$ is invertible for $\pi/4 < \arg\lambda < \pi/2$, $|\lambda|\gg 1$. Since $R_\epsilon(\lambda): L^2 \to H^2$ is analytic in $\{ \lambda :  -\pi/8 < \arg\lambda <7\pi/8 \}$, see \eqref{Davies spectrum}, $V: H^2 \to L^2 $ is compact by \eqref{V decay}, we have $\lambda\mapsto V R_\epsilon(\lambda)$ is an analytic family of compact operators for $-\pi/8 < \arg\lambda <7\pi/8$. Hence $\lambda\mapsto ( I + V R_\epsilon(\lambda) )^{-1}$ is a meromorphic family of operators in the same range of $\lambda$. Using \eqref{eqn:RVeps}, we conclude that  $R_{V,\epsilon}(\lambda) = R_\epsilon(\lambda) ( I + V R_\epsilon(\lambda) )^{-1}$ is meromorphic for $-\pi/8 < \arg\lambda <7\pi/8$ (in fact $R_{V,\epsilon}(\lambda)$ is meromorphic for $\lambda\in\CC$ by the Gohberg--Sigal factorization theorem - see \cite[\S C.4]{res}), with poles $\{\lambda_j(\epsilon)\}_{j=1}^\infty$, i.e. $\{ \lambda_j(\epsilon)^2 \}_{j=1}^\infty$ are the eigenvalues of $P_\epsilon$ in $\{z\in\CC : \arg z \neq -\pi/4\}$. Then we have
\begin{lem}
\label{lem:newmeroCont}
    For each $\epsilon > 0$,
    \[
        \lambda \mapsto ( I + \sqrt{V}R_\epsilon(\lambda)\sqrt{V} )^{-1},\quad -\pi/8 < \arg\lambda < 7\pi/8,
    \]
    is a meromorphic family of operators on $L^2(\RR^n)$ with poles of finite rank. Moreover,
    \begin{equation}
    \label{eqn:mult 1}
        m_\epsilon(\lambda) := \frac{1}{2\pi i} \tr \oint_\lambda ( I + \sqrt{V}R_\epsilon(\zeta)\sqrt{V} )^{-1} \partial_\zeta(\sqrt{V}R_\epsilon(\zeta)\sqrt{V})\,d\zeta ,
    \end{equation}
    where the integral is over a positively oriented circle enclosing $\lambda$ and containing no poles other than possibly $\lambda$, satisfies
    \begin{equation}
    \label{eqn:mult 2}
        m_\epsilon(\lambda) = \frac{1}{2\pi i} \tr \oint_\lambda (\zeta^2 - P_\epsilon)^{-1} 2\zeta\, d\zeta .
    \end{equation}
\end{lem}

\noindent
{\bf Remark.} The multiplicity of an eigenvalue $\lambda^2$ of $P_\epsilon$ can be defined by the right side of \eqref{eqn:mult 2}, thus Lemma \ref{lem:newmeroCont} implies that the poles of $( I + \sqrt{V}R_\epsilon(\lambda)\sqrt{V} )^{-1}$ coincide with $\{\lambda_j(\epsilon)\}_{j=1}^\infty$ given in \eqref{PepsEigenvalues}, with agreement of multiplicities.

\begin{proof}
Following the above argument, it easy to see that $\lambda\mapsto \sqrt{V}R_\epsilon(\lambda)\sqrt{V}$ is an analytic family of compact operators for $-\pi/8 < \arg \lambda < 7\pi/8$. Then
\[ \lambda \mapsto ( I + \sqrt{V}R_\epsilon(\lambda)\sqrt{V} )^{-1},\quad -\pi/8 < \arg\lambda < 7\pi/8, \]
is a meromorphic family of operators, since $ I + \sqrt{V}R_\epsilon(\lambda)\sqrt{V} $ is invertible for $\pi/4 < \arg\lambda <\pi/2$, $|\lambda|\gg 1$ by \eqref{Reps bound}. In this range of $\lambda$, $I + V R_\epsilon(\lambda)$ is also invertible by the Neumann series argument, thus we have
\begin{equation}
\label{eqn:newmerocont}
\begin{split}
    (P_\epsilon - \lambda^2 )^{-1} & = R_\epsilon(\lambda) ( I +  V R_\epsilon(\lambda) )^{-1}  \\
    & = R_\epsilon(\lambda)\sum_{j=0}^\infty (-1)^j ( V R_\epsilon(\lambda) )^j  \\
    & = R_\epsilon(\lambda) ( I - \sqrt{V} \,\sum_{j=0}^\infty (-1)^j (\sqrt{V} R_\epsilon(\lambda) \sqrt{V} )^j\, \sqrt{V} R_\epsilon(\lambda) )  \\
    & = R_\epsilon(\lambda) [\, I - \sqrt{V} ( I + \sqrt{V} R_\epsilon(\lambda) \sqrt{V} )^{-1} \sqrt{V} R_\epsilon(\lambda) \,].
\end{split}
\end{equation}
Since both sides of \eqref{eqn:newmerocont} are meromorphic for $-\pi/8 < \arg\lambda <7\pi/8 $, by meromorphy, we conclude that \eqref{eqn:newmerocont} holds for all $-\pi/8 < \arg\lambda <7\pi/8 $, as an identity between meromorphic families of operators.

To obtain the multiplicity formula, we fix any $\lambda$ with $-\pi/8 < \arg\lambda <7\pi/8$, then there exists a neighborhood $\lambda\in U$ in this half plane and finite rank operators $A_j$, $1\leq j\leq J$ such that $ (  I + \sqrt{V}R_\epsilon(\zeta) \sqrt{V} )^{-1} - \sum_{j=1}^J \frac{A_j}{(\zeta - \lambda)^j}$ is analytic in $ \zeta \in U$.
Let $\mathcal{C}_\lambda \subset U$ be a positively oriented circle enclosing $ \lambda $
and containing no poles of $(  I + \sqrt{V}R_\epsilon(\zeta) \sqrt{V} )^{-1}$ other than possibly $ \lambda $, thus it also contains no poles of $( \zeta^2 - P_\epsilon )^{-1}$ other than possibly $ \lambda $ as a consequence of \eqref{eqn:newmerocont}. On the one hand, we can compute
\begin{equation}
\label{eqn:multproof1}
\begin{split}
    m_{\epsilon} ( \lambda ) & = \frac{ 1}{ 2 \pi i}\tr\int_{\mathcal{C}_\lambda}   ( I + \sqrt{V}R_\epsilon(\zeta) \sqrt{V} )^{-1} \sqrt{V}R_\epsilon(\zeta)^2 \sqrt{V} \,2\zeta d\zeta \\
    & = \frac{ 1}{ 2 \pi i}\tr\int_{\mathcal{C}_\lambda} \sum_{j=1}^J \frac{A_j \sqrt{V} R_\epsilon(\zeta)^2 2\zeta \sqrt{V}}{(\zeta - \lambda)^j} d\zeta \\
    & = \sum_{j=1}^J \sum_{k=0}^{j-1} \frac{1}{k!(j-1-k)!} \tr A_j \sqrt{V} \, \partial_\zeta^k R_\epsilon(\zeta)\,
    \partial_\zeta^{j-1-k} (R_\epsilon(\zeta)2\zeta) \sqrt{V}.
\end{split}
\end{equation}
On the other hand, by \eqref{eqn:newmerocont}, we have
\begin{equation}
\label{eqn:multproof2}
\begin{split}
    { }&\quad\  \frac{1}{ 2 \pi i } \tr \oint_\lambda ( \zeta^2 - P_\epsilon) )^{-1} 2\zeta d\zeta  \\
    & = \frac{ 1}{ 2 \pi i}\tr\int_{\mathcal{C}_\lambda} \sum_{j=1}^J \frac{R_\epsilon(\zeta) 2\zeta \sqrt{V} A_j \sqrt{V} R_\epsilon(\zeta) }{(\zeta - \lambda)^j} d\zeta \\
    & = \sum_{j=1}^J \sum_{k=0}^{j-1} \frac{1}{k!(j-1-k)!} \tr \partial_\zeta^{j-1-k} (R_\epsilon(\zeta)2\zeta) \sqrt{V} A_j \sqrt{V} \, \partial_\zeta^k R_\epsilon(\zeta) .
\end{split}
\end{equation}
Now we compare \eqref{eqn:multproof1} and \eqref{eqn:multproof2}, since each $A_j$ has finite rank, we can apply cyclicity of the trace to obtain the multiplicity formula \eqref{eqn:mult 2}.
\end{proof}

\section{Proof of convergence}
\label{poc}

The proof of convergence is based on Lemma \ref{lem:meroCont}, Lemma \ref{lem:newmeroCont}, with an application of the Gohberg--Sigal--Rouch\'e theorem, see Gohberg--Sigal \cite{gohberg1971operator} and \cite[Appendix C.]{res}.

We now state a more precise version of Theorem \ref{thm:1} involving the multiplicities given in \eqref{resonance mult} and \eqref{eqn:mult 1} as follows:
\begin{thm}
\label{thm:2}
For any $\Omega$ given in \eqref{rectangle Omega}, there exists $\delta_0=\delta_0(\Omega)$ satisfying the following: for any $0<\delta<\delta_0$, there exists $\epsilon_\delta > 0$ such that for any  $\lambda\in\Omega$ with $m(\lambda)>0$,
\[ \#\, \{\lambda_j(\epsilon)\}_{j=1}^\infty \cap B(\lambda,\delta) = m(\lambda),\quad\textrm{for all }0<\epsilon<\epsilon_\delta, \]
where $\{\lambda_j(\epsilon)\}_{j=1}^\infty$ given in \eqref{PepsEigenvalues} is counted with multiplicity, $B(\lambda,\delta):=\{z\in\CC:|z-\lambda|<\delta\}$.
\end{thm}

\begin{proof}
In view of Lemma \ref{lem:meroCont}, the poles of $( I + \sqrt{V}R_0(\lambda)\sqrt{V} )^{-1}$ are isolated in the region $\{\lambda\in\CC : \Re\lambda>0,\Im\lambda>-\gamma\}$, thus there are finitely many $\lambda\in\Omega$ with $m(\lambda)>0$, denoted by $\lambda_1,\ldots,\lambda_J$. We choose $\delta_0 > 0$ such that $B(\lambda_j,\delta_0)$, $j=1,\ldots,J$ are disjoint discs in $\Omega$, then for any fixed $0<\delta<\delta_0$ and each $\lambda\in\Omega$ with $m(\lambda)>0$, we have
\[ \|( I + \sqrt{V}R_0(\zeta)\sqrt{V} )^{-1} \|_{L^2\to L^2} < C(\delta),\quad \forall\,\zeta\in\partial B(\lambda,\delta), \]
for some constant $C(\delta)>0$.

In order to apply the Gohberg--Sigal--Rouch\'e theorem, we need to estimate :
\[ \| I + \sqrt{V}R_\epsilon(\zeta)\sqrt{V} - ( I + \sqrt{V}R_0(\zeta)\sqrt{V} ) \|_{L^2\to L^2},\quad\textrm{for any }\zeta\in \Omega. \]
1. Choose $\chi\in \CIc(\RR^n)$ satisfying $\chi\equiv 1$ in $B_{\RR^n}(0,1)$ and $\supp\chi\subset B_{\RR^n}(0,2)$, here $B_{\RR^n}(0,r) := \{ x\in\RR^n : |x|< r \}$, we define $\chi_R (x) = \chi(R^{-1} x)$ and calculate:
\begin{equation}
\label{eqn:poc1}
    \begin{split}
    { }&\quad\  I + \sqrt{V}R_\epsilon(\zeta)\sqrt{V} - ( I + \sqrt{V}R_0(\zeta)\sqrt{V} ) \\
    & = \sqrt{V}R_\epsilon(\zeta)\sqrt{V} - \chi_R\sqrt{V}R_\epsilon(\zeta)\chi_R\sqrt{V} + \sqrt{V}\chi_R(R_\epsilon(\zeta)-R_0(\zeta))\chi_R\sqrt{V} \\
    & \quad - (\sqrt{V}R_0(\zeta)\sqrt{V} - \chi_R\sqrt{V}R_0(\zeta)\chi_R\sqrt{V}) .
    \end{split}
\end{equation}
2. The first term can be written as $(1-\chi_R)\sqrt{V}R_\epsilon(\zeta)\sqrt{V} + \chi_R\sqrt{V}R_\epsilon(\zeta)(1-\chi_R)\sqrt{V}$. Let $\Tilde\gamma=(\gamma + \gamma')/2$, then
\[ (1-\chi_R)\sqrt{V}R_\epsilon(\zeta)\sqrt{V} = (1-\chi_R)\sqrt{V}e^{\Tilde{\gamma}|x|}(e^{-\Tilde{\gamma}|x|}R_\epsilon(\zeta))e^{-\Tilde{\gamma}|x|})\sqrt{V}e^{\Tilde{\gamma}|x|}, \]
where $|\sqrt{V(x)}e^{\Tilde{\gamma}|x|}|\leq C e^{(\Tilde{\gamma}-\gamma)|x|} = C e^{-(\gamma-\gamma')|x|/2}$. By Lemma \ref{lem:weighted Davies resolvent estimate}, there exists $\epsilon_0=\epsilon_0(\Omega,\Tilde{\gamma})$ such that for any $0<\epsilon<\epsilon_0$, $\|e^{-\Tilde{\gamma}|x|}R_\epsilon(\zeta))e^{-\Tilde{\gamma}|x|}\|_{L^2\to L^2}\leq C(\Omega,\Tilde{\gamma}) $. Thus,
\[ \|(1-\chi_R)\sqrt{V}R_\epsilon(\zeta)\sqrt{V} \|_{L^2\to L^2} \leq C(\Omega,\gamma)e^{-(\gamma-\gamma')R/2},\quad\textrm{for any }0<\epsilon<\epsilon_0 .\]
Similarly, we can bound $\|\chi_R\sqrt{V}R_\epsilon(\zeta)(1-\chi_R)\sqrt{V}\|_{L^2\to L^2}$ by the right side above. Hence for any $0<\epsilon<\epsilon_0$,
\begin{equation}
\label{estimate1:poc}
    \|\sqrt{V}R_\epsilon(\zeta)\sqrt{V} - \chi_R\sqrt{V}R_\epsilon(\zeta)\chi_R\sqrt{V}\|_{L^2\to L^2}\leq Ce^{-(\gamma-\gamma')R/2},\quad\forall\,\zeta\in\Omega.
\end{equation}
3. We can estimate the third term in \eqref{eqn:poc1} by a similar argument. \eqref{weighted R0 bounds} implies that
\[ \| e^{-\Tilde{\gamma}|x|}R_0(\zeta)e^{-\Tilde{\gamma}|x|} \|_{L^2\to L^2} \leq C(\Omega,\gamma),\quad\forall\,\zeta\in\Omega. \]
Hence, arguing as above, we obtain that
\begin{equation}
\label{estimate2:poc}
    \|\sqrt{V}R_0(\zeta)\sqrt{V} - \chi_R\sqrt{V}R_0(\zeta)\chi_R\sqrt{V}\|_{L^2\to L^2}\leq Ce^{-(\gamma-\gamma')R/2},\quad\forall\,\zeta\in\Omega.
\end{equation}
4. We note that
\[ \chi_R (R_\epsilon(\zeta)-R_0(\zeta)) \chi_R = i\epsilon\, \chi_R (-\Delta -i\epsilon x^2 -\zeta^2)^{-1} x^2 (\Delta-\zeta^2)^{-1} \chi_R , \]
and recall \cite{Zw-vis} that there exists $C=C(\Omega,\chi_R)$ (independent of $\epsilon$) such that
\[ \|\chi_R (-\Delta -i\epsilon x^2 -\zeta^2)^{-1} x^2 (\Delta-\zeta^2)^{-1} \chi_R\|_{L^2\to L^2}\leq C,\quad\forall\,\zeta\in\Omega,\,\epsilon>0, \]
which is proved using the method of complex scaling, see \cite[\S 5]{Zw-vis} for details. Hence
\begin{equation}
\label{estimate3:poc}
    \|\sqrt{V}\chi_R(R_\epsilon(\zeta)-R_0(\zeta))\chi_R\sqrt{V}\|_{L^2\to L^2} \leq C(\Omega,\chi_R)\,\epsilon,\quad\forall\,\zeta\in\Omega,\,\epsilon>0.
\end{equation}
By \eqref{estimate1:poc} and \eqref{estimate2:poc}, we can first fix $R$ sufficiently large such that
\[
    \|\sqrt{V}R_\epsilon(\zeta)\sqrt{V} - \chi_R\sqrt{V}R_\epsilon(\zeta)\chi_R\sqrt{V}\|_{L^2\to L^2}\leq 1/(3C(\delta)),\quad\forall\,\zeta\in\Omega,\ 0\leq \epsilon<\epsilon_0.
\]
Then by \eqref{estimate3:poc}, there exists $\epsilon_\delta>0$ such that for all $0<\epsilon<\epsilon_\delta$,
\[ \|\sqrt{V}\chi_R(R_\epsilon(\zeta)-R_0(\zeta))\chi_R\sqrt{V}\|_{L^2\to L^2} \leq 1/(3C(\delta)),\quad\forall\,\zeta\in\Omega. \]
We may assume that $\epsilon_\delta<\epsilon_0$, thus by \eqref{eqn:poc1}, we conclude that for each $0<\epsilon<\epsilon_\delta$,
\[ \| ( I + \sqrt{V}R_0(\zeta)\sqrt{V} )^{-1} (\, I + \sqrt{V}R_\epsilon(\zeta)\sqrt{V} - ( I + \sqrt{V}R_0(\zeta)\sqrt{V} )\,) \|_{L^2\to L^2} < 1,
\]
on $\partial B(\lambda,\delta)$.

Now we apply the Gohberg--Sigal--Rouch\'e theorem to obtain that
\[
\begin{split}
    m(\lambda) & = \frac{1}{2\pi i} \tr \int_{\partial B(\lambda,\delta)} ( I + \sqrt{V}R_0(\zeta)\sqrt{V} )^{-1} \partial_\zeta(\sqrt{V}R_0(\zeta)\sqrt{V})\,d\zeta \\
    & = \frac{1}{2\pi i} \tr \int_{\partial B(\lambda,\delta)} ( I + \sqrt{V}R_\epsilon(\zeta)\sqrt{V} )^{-1} \partial_\zeta(\sqrt{V}R_\epsilon(\zeta)\sqrt{V})\,d\zeta,
\end{split}
\]
for each $0<\epsilon<\epsilon_\delta$. Let $\lambda_1(\epsilon),\ldots,\lambda_K(\epsilon)$ be the distinct poles of $( I + \sqrt{V}R_\epsilon(\zeta)\sqrt{V} )^{-1}$ in $B(\lambda,\delta)$, then
\[ m(\lambda) = \sum_{k=1}^K \frac{1}{2\pi i} \tr \oint_{\lambda_k(\epsilon)} ( I + \sqrt{V}R_\epsilon(\zeta)\sqrt{V} )^{-1} \partial_\zeta(\sqrt{V}R_\epsilon(\zeta)\sqrt{V})\,d\zeta = \sum_{k=1}^K m_\epsilon(\lambda_k(\epsilon)), \]
Therefore, with Lemma \ref{lem:newmeroCont} and \eqref{eqn:mult 2}, we obtain that
\[ \#\,\{\lambda_j(\epsilon)\}_{j=1}^\infty \cap B(\lambda,\delta) = m(\lambda),\quad\forall\,0<\epsilon<\epsilon_\delta, \]
which completes the proof.
\end{proof}

% Tables may be be put in the text as floats.
% Here is an example of the general form of a table:
% Fill in the caption in the braces of the \caption{} command. Put the label
% that you will use with \ref{} command in the braces of the \label{} command.
% Insert the column specifiers (l, r, c, d, etc.) in the empty braces of the
% \begin{tabular}{} command.
%
% \begin{table}
% \caption{\label{} }
% \begin{tabular}{}
% \end{tabular}
% \end{table}

% If you have acknowledgments, this puts in the proper section head.
\begin{acknowledgments}
% Put your acknowledgments here.
The author would like to thank
Maciej Zworski for helpful discussions. I am also grateful to the anonymous referee for the careful reading
of the first version and for the valuable comments. This project was supported in part by the National Science Foundation grant 1500852.
\end{acknowledgments}

\medskip
\noindent
{\bf DATA AVAILABILITY STATEMENT} \\
The data that supports the findings of this study are available within the article.

% Create the reference section using BibTeX:
%\bibliographystyle{unsrt}
%\bibliography{JMP20-AR-01038}

\begin{thebibliography}{10}

\bibitem{RiMe}
UV~Riss and H-D Meyer.
\newblock Reflection-free complex absorbing potentials.
\newblock {\em Journal of Physics B: Atomic, Molecular and Optical Physics},
  28:1475, (1995).

\bibitem{semi}
Tamar Seideman and William~H Miller.
\newblock Calculation of the cumulative reaction probability via a discrete
  variable representation with absorbing boundary conditions.
\newblock {\em The Journal of chemical physics}, 96:4412, (1992).

\bibitem{Jag}
Thomas-C Jagau, Dmitry Zuev, Ksenia~B Bravaya, Evgeny Epifanovsky, and Anna~I
  Krylov.
\newblock A fresh look at resonances and complex absorbing potentials: Density
  matrix-based approach.
\newblock {\em The journal of physical chemistry letters}, 5:310, (2014).

\bibitem{Zw-vis}
Maciej Zworski.
\newblock Scattering resonances as viscosity limits.
\newblock In {\em Algebraic and Analytic Microlocal Analysis}, page 635.
  Springer, (2013).

\bibitem{xiong2020}
Haoren Xiong.
\newblock Resonances as viscosity limits for exterior dilation analytic
  potentials.
\newblock {\em arXiv preprint arXiv:2002.12490}, (2020).

\bibitem{hunziker1986}
Walter Hunziker.
\newblock Distortion analyticity and molecular resonance curves.
\newblock In {\em Annales de l'IHP Physique th{\'e}orique}, volume~45, page
  339, (1986).

\bibitem{sjostrand1991}
Johannes Sj{\"o}strand and Maciej Zworski.
\newblock Complex scaling and the distribution of scattering poles.
\newblock {\em Journal of the American Mathematical Society}, 4:729, (1991).

\bibitem{nakamura1990}
Shu Nakamura.
\newblock Distortion analyticity for two-body schr{\"o}dinger operators.
\newblock In {\em Annales de l'IHP Physique th{\'e}orique}, volume~53, page
  149, (1990).

\bibitem{Nakamura}
Kentaro Kameoka and Shu Nakamura.
\newblock Resonances and viscosity limit for the wigner-von neumann type
  hamiltonian.
\newblock {\em arXiv preprint arXiv:2003.07001}, (2020).

\bibitem{froese}
Richard Froese.
\newblock Upper bounds for the resonance counting function of schr{\"o}dinger
  operators in odd dimensions.
\newblock {\em Canadian Journal of Mathematics}, 50:538, (1998).

\bibitem{gannot}
Oran Gannot.
\newblock From quasimodes to resonances: exponentially decaying perturbations.
\newblock {\em Pacific Journal of Mathematics}, 277:77, (2015).

\bibitem{res}
Semyon Dyatlov and Maciej Zworski.
\newblock {\em Mathematical theory of scattering resonances}, volume 200.
\newblock American Mathematical Soc., (2019).

\bibitem{Dav}
E~Brian Davies.
\newblock Pseudo--spectra, the harmonic oscillator and complex resonances.
\newblock {\em Proceedings of the Royal Society of London. Series A:
  Mathematical, Physical and Engineering Sciences}, 455:585, (1999).

\bibitem{HSV}
Michael Hitrik, Johannes Sj{\"o}strand, and Joe Viola.
\newblock Resolvent estimates for elliptic quadratic differential operators.
\newblock {\em Analysis \& PDE}, 6:181, (2013).

\bibitem{Zw}
Maciej Zworski.
\newblock {\em Semiclassical analysis}, volume 138.
\newblock American Mathematical Soc., (2012).

\bibitem{gohberg1971operator}
IC~U. Gohberg and E.~I. Sigal.
\newblock An operator generalization of the logarithmic residue theorem and the
  theorem of rouch{\'e}.
\newblock {\em Mathematics of the USSR-Sbornik}, 13:603, (1971).

\end{thebibliography}

\end{document}